\documentclass[a4paper,11pt]{amsart}
\pdfoutput=1

\usepackage[utf8]{inputenc}
\usepackage{mathtools,amssymb}
\usepackage{hyperref}
\usepackage{xcolor,graphicx}
\usepackage{mathrsfs}
\usepackage{enumitem}

\graphicspath{{images/}}

\newtheorem{theorem}{Theorem}[section]
\newtheorem{lemma}[theorem]{Lemma}

\newtheorem{remark}[theorem]{Remark}

\makeatletter
\@namedef{subjclassname@2020}{\textup{2020} Mathematics Subject Classification}
\makeatother
\title[The geodesic ray transform on reversible Finsler manifolds]{The geodesic ray transform on spherically symmetric reversible Finsler manifolds}

\author{Joonas Ilmavirta}
\thanks{Department of Mathematics and Statistics, University of Jyv\"askyl\"a, P.O. Box 35 (MaD) FI-40014 University of Jyv\"askyl\"a, Finland; \href{mailto:joonas.ilmavirta@jyu.fi}{joonas.ilmavirta@jyu.fi}}
\author{Keijo M\"onkk\"onen}
\thanks{Department of Mathematics and Statistics, University of Jyv\"askyl\"a, P.O. Box 35 (MaD) FI-40014 University of Jyv\"askyl\"a, Finland; \href{mailto:keijo.m.t.monkkonen@jyu.fi}{keijo.m.t.monkkonen@jyu.fi}}

\date{\today}

\usepackage{graphicx}

\newcommand{\R}{{\mathbb R}}
\newcommand{\Z}{{\mathbb Z}}

\newcommand{\der}{{\mathrm d}}

\newcommand{\geod}{\mathcal{I}}
\newcommand{\bigoh}{\mathcal{O}}

\newcommand{\NTR}[1]{}

\newcommand{\abel}{\mathcal{A}}

\newcommand{\abs}[1]{\left\lvert #1 \right\rvert}
\newcommand{\aabs}[1]{\left\lVert #1 \right\rVert}

\begin{document}
\maketitle

\begin{abstract}
We show that the geodesic ray transform is injective on scalar functions on spherically symmetric reversible Finsler manifolds where the Finsler norm satisfies a Herglotz condition. We use angular Fourier series to reduce the injectivity problem to the invertibility of generalized Abel transforms and by Taylor expansions of geodesics we show that these Abel transforms are injective. Our result has applications in linearized boundary rigidity problem on Finsler manifolds and especially in linearized elastic travel time tomography.

\end{abstract}

\section{Introduction}
\NTR{We have corrected the grammatical errors and typos which were pointed out by the referees.}
In this paper we study the following mathematical inverse problem arising in integral geometry: if we know the integrals of a scalar function~$f$ over all geodesics of a Finsler manifold $(M, F)$, can we determine $f$? 
Since the problem is linear, we can formulate it in terms of the kernel of the\NTR{Added "of the kernel".} geodesic ray transform~$\geod$: if $\geod f(\gamma)=0$ for all geodesics~$\gamma$, does it follow that $f=0$? In other words, is the geodesic ray transform~$\geod$ injective on scalar fields?
This inverse problem (and its generalization to tensor fields) has been usually studied on Riemannian manifolds and a variety of results under different types of assumptions is known in the Riemannian setting~\cite{IM:integral-geometry-review, PSU-tensor-tomography-progress, SHA-integral-geometry-tensor-fields}. We show that in the case of spherical symmetry, reversibility and a Herglotz condition the answer is positive for Finsler manifolds as well:~$\geod$ is injective on scalar fields.

A Finsler norm~$F$ is a non-negative function $F\colon TM\to [0, \infty)$ such that for every $x\in M$ the map $y\mapsto F(x, y)$ defines a Minkowski norm in~$T_xM$ (see section~\ref{sec:finslermanifolds} for details). 
We focus on spherically symmetric and reversible Finsler norms. We show that if $M\subset\R^n$ is an annulus centered at the origin and~$F$ is a spherically symmetric reversible Finsler norm on~$M$ which satisfies the Herglotz condition (see equation~\eqref{eq:herglotz} and section~\ref{sec:herglotz}), then the geodesic ray transform~$\geod$ is injective on $L^2$-functions (see section~\ref{sec:maintheorem} and theorem~\ref{thm:maintheorem}). This generalizes earlier Riemannian results in~\cite{deI:abel-transforms-x-ray-tomography, RO-herglotz-linearized} to the Finslerian case and our theorem can be seen as a Helgason-type support theorem on Finsler manifolds (see e.g.~\cite{deI:abel-transforms-x-ray-tomography, HE:integral-geometry-radon-transforms}). An example of a non-Riemannian geometry where our main theorem applies is a Finsler norm arising from an anisotropic sound speed which is reversible and spherically symmetric and satisfies the Herglotz condition (see section \ref{sec:herglotz}).\NTR{Added this sentence.}

We use angular Fourier series to reduce the inverse problem to the invertibility of certain Abel-type integral transforms. This approach was used in~\cite{deI:abel-transforms-x-ray-tomography} where the authors 
proved various injectivity results of generalized Abel transforms\NTR{Rephrased this sentence.}
which we also use in the proof of our main result. By a careful treatment of the Taylor expansions of geodesics near their lowest point to the origin we show that the Abel transforms we encounter are indeed injective.

Our result is related to the travel time tomography or the boundary rigidity problem. Travel time tomography is an imaging method used in seismology where one wants to determine the speed of sound inside the Earth by measuring travel times of seismic waves on the surface of the Earth~\cite{SUVZ-travel-time-tomography}. The ray paths correspond to geodesics and travel times to lengths of geodesics. The boundary rigidity problem is a more general geometric inverse problem where one wants to determine a Riemannian metric or more generally a Finsler norm from the distances between boundary points~\cite{SUVZ-travel-time-tomography}. 

The travel time tomography problem was already solved in the 1900s for radial sound speeds satisfying the Herglotz condition~\cite{HE-inverse-kinematic-problem, WZ-kinematic-problem}. However, it is observed that the Earth exhibits more complicated and especially anisotropic behaviour with respect to the sound speed~\cite{CRE-anisotropy-inner-core, DA-PREM-model, SHE-introduction-to-seismology}. In the anisotropic case seismic rays propagate along geodesics of a Finsler norm~\cite{ABS-seismic-rays-as-finsler-geodesics, YN-finsler-seismic-ray-path} and Riemannian geometry is not enough to describe the most general types of anisotropies. The boundary rigidity problem is already a difficult non-linear inverse problem in the Riemannian case and anisotropies complicate things even more. 

It is known that Finsler norms arising in elasticity are reversible~\cite{deILS-finsler-boundary-distance-map} (see also section~\ref{subsec:elastictraveltime}) which puts some constraints on the geometry. Invariance under rotations is a natural physical requirement for the Finsler norm (or sound speed) since the Earth is (roughly) spherically symmetric. Our Herglotz condition~\eqref{eq:herglotz} is a natural generalization of the usual Herglotz condition to anisotropic sound speeds (see equation~\eqref{eq:herglotzanisotropicspeed}) and it implies that certain geodesics behave nicely (see section~\ref{sec:herglotz}). We can further simplify the problem by linearizing it. If the variations of the Riemannian metric or Finsler norm are conformal, then linearization of the boundary rigidity problem leads to the geodesic ray transform of scalar functions on the base manifold (see~\cite{SHA-integral-geometry-tensor-fields} and section~\ref{sec:conformal}). This especially holds for a family of conformal Finsler norms induced by a conformal family of stiffness tensors. 

Our main theorem implies boundary rigidity up to first order for a conformal family of spherically symmetric reversible Finsler norms satisfying the Herglotz condition. In terms of elasticity, if we have a conformal family of stiffness tensors (a family of factorized anisotropic inhomogeneous media~\cite{CE-seismic-ray-theory, YN-finsler-seismic-ray-path}) such that the induced family of Finsler norms give the same distances between boundary points and satisfy the assumptions of theorem~\ref{thm:maintheorem}, then the stiffness tensors are equal up to first order (see section~\ref{subsec:elastictraveltime}).

\subsection{The main theorem}
\label{sec:maintheorem}
Let us first quickly introduce the key definitions and notation. More details can be found in section~\ref{sec:preliminaries}.

Let~$M$ be a smooth manifold. A Finsler norm $F\colon TM\to [0, \infty)$ is a non-negative function on the tangent bundle~$TM$ so that the map $y\mapsto F(x, y)$ is a positively homogeneous (but not necessarily homogeneous) norm in the tangent space $T_xM$ for each $x\in M$. Finsler norm~$F$ is reversible if $F(x, -y)=F(x, y)$ for all $x\in M$ and $y\in T_x M$. A reversible Finsler norm defines a homogeneous norm in~$T_x M$. The length of a curve $\gamma\colon [a, b]\to M$ is defined as $L(\gamma)=\int_a^b F(\gamma(t), \dot{\gamma}(t))\der t$. The geodesics of a Finsler norm are critical points of the length functional $\gamma\mapsto L(\gamma)$, or equivalently they satisfy the geodesic equation (see equation~\eqref{eq:geodesicequation}).

Our manifold will eventually be an annulus $M\subset\R^2$ with outer boundary centered at the origin. The inner boundary of the annulus is not included in~$M$ so~$\partial M$ only consists of the outer boundary. We say that a Finsler norm~$F$ on~$M$ is spherically symmetric, if $U^*F=F$ for all $U\in SO(2)$. Let $(x^1, x^2)=(r, \theta)$ be the polar coordinates on~$M$. These coordinates induce a coordinate basis $\{\partial_r, \partial_\theta\}$ in every tangent space~$T_{(r, \theta)}M$. If $y=y^1\partial_r+y^2\partial_\theta\in T_{(r, \theta)}M$, then we denote its coordinates by $(y^1, y^2)=(\rho, \phi)$; see figure~\ref{fig:coordinates}. We can then equivalently say that the Finsler norm~$F$ on~$M$ is spherically symmetric if it is independent of the angular variable $\theta$, i.e. $F=F(r, \rho, \phi)$.\NTR{Added this sentence.}

We say that a spherically symmetric reversible Finsler norm~$F=F(r, \rho, \phi)$ on~$M$ satisfies the Herglotz condition, if 
\begin{equation}
\label{eq:herglotz}
\partial_r F^2(r, 0, \phi)>0
\end{equation}
for all $r\in (R, 1]$ and $\phi\neq 0$.
Here $R\in (0, 1)$ is the inner radius of the annulus.
The Herglotz condition implies that $(M, F)$ admits a strictly convex foliation and geodesics, which are initially tangential to circles, reach the outer boundary in finite time (see lemmas~\ref{lemma:foliation} and~\ref{lemma:non-trapping}). More generally, we say that a spherically symmetric reversible Finsler norm~$F$ on an $n$-dimensional annulus $M\subset\R^n$ satisfies the Herglotz condition if~$F$ satisfies the two-dimensional Herglotz condition~\eqref{eq:herglotz} on all slices $M\cap P$ where $P\subset\R^n$ is a two-dimensional subspace. 

\begin{figure}[htp]
\centering
\includegraphics[height=8cm]{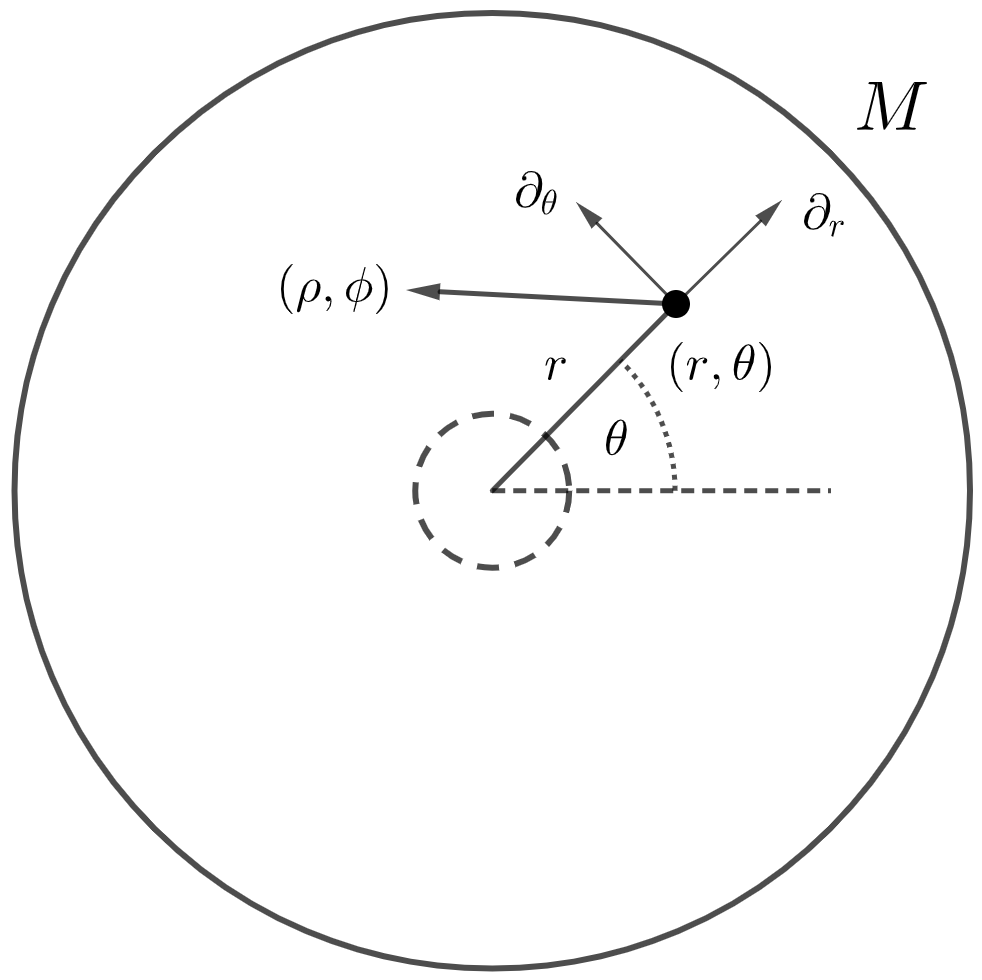}
\caption{Our manifold~$M$ and the coordinate system $(r, \theta, \rho, \phi)$ on~$TM$. The coordinate vector fields~$\partial_r$ and~$\partial_\theta$ form a basis in each tangent space~$T_{(r, \theta)}M$ and the coordinates of $y\in T_{(r, \theta)}M$ with respect to this basis are $(\rho, \phi)$. The components are given by $\rho=\der r(y)$ and $\phi=\der\theta(y)$. We use polar coordinates only on~$M$; the induced coordinates on~$T_{(r, \theta)}M$ are Euclidean. The inner boundary (dashed) is not included in~$M$.}
\label{fig:coordinates}
\end{figure}

The geodesic ray transform~$\geod$ takes a sufficiently regular scalar field~$f$ on~$M$ and integrates it over geodesics, i.e. $\geod f(\gamma)=\int_\gamma f\der s$ where~$\gamma$ is a geodesic of the Finsler norm~$F$. 
The Herglotz condition guarantees that geodesics which are initially tangential to circles 
have a unique closest point to the origin and the integrals exist for such geodesics.\NTR{Changed "all geodesics" to "geodesics which are initially tangential to circles" since we do not know if the manifold $(M, F)$ is non-trapping.}

Our main theorem is the following injectivity result. The proof of the theorem can be found in section~\ref{sec:proofofmainresult}.

\begin{theorem}
\label{thm:maintheorem}
Let $n\geq 2$, $M=\bar{B}(0, 1)\setminus\bar{B}(0, R)\subset\R^n$ where $R\in (0, 1)$ and equip~$M$ with a smooth spherically symmetric reversible Finsler norm~$F$ which satisfies the Herglotz condition. Then the geodesic ray transform~$\geod$ is injective on~$L^2(M)$.
\end{theorem}

\begin{remark}
\label{remark:reduction}
It is enough to prove theorem~\ref{thm:maintheorem} in two dimensions. Namely, if we intersect a higher-dimensional annulus~$M\subset\R^n$ with any two-dimensional linear subspace $P\subset\R^n$, we get a totally geodesic submanifold~$M\cap P\subset M$ since~$F$ is reversible and spherically symmetric. Also, by~\cite[Lemma 17]{deI:abel-transforms-x-ray-tomography} it holds that if $f\in L^2(M)$, then $f|_{M\cap P}\in L^2(M\cap P)$ for almost every two-dimensional plane~$P$. Hence if theorem~\ref{thm:maintheorem} is true for $n=2$, then it is also true for all $n\geq 2$.
\end{remark}

\begin{remark}
\label{remark:regularity}
We assume that the Finsler norm~$F$ in theorem~\ref{thm:maintheorem} is smooth. This regularity assumption could be weakened: from the proof of theorem~\ref{thm:maintheorem} one sees that a finite number of derivatives with respect to the variables~$x\in M$ and~$y\in T_xM$ is enough. However, we are not going to quantify or optimize the needed regularity assumptions in this paper.
\end{remark}

In theorem~\ref{thm:maintheorem} we assume that the Finsler norm is reversible since it simplifies the proof and Finsler norms arising in elasticity are reversible. Our main application of theorem~\ref{thm:maintheorem} is the seismic imaging of the Earth and therefore we let~$M$ to be an annulus and~$F$ to be spherically symmetric.\NTR{Moved this sentence here from the end of this paragraph.} We could formulate theorem~\ref{thm:maintheorem} in terms of a general family of curves satisfying certain properties (see remark~\ref{remark:familyofcurves}).
In fact, in the proof of theorem~\ref{thm:maintheorem} we only use integrals of~$f$ over geodesics which have a unique lowest point to the origin. Due to the Herglotz condition geodesics can not have more than one point where the radial speed~$\dot{r}$ vanishes. However, we do not know whether our manifold is non-trapping, i.e. we do not know if all geodesics reach the boundary in finite time or if there exists trapped geodesics.\NTR{Rewrote the discussion on the geodesics we use in our main theorem.}

Theorem~\ref{thm:maintheorem} is proved in the following way (see sections~\ref{sec:preliminaries}, \ref{sec:proofofmainresult} and~\ref{sec:regularityofkernel} for more details). Since our manifold is annulus we can express any $L^2$-function~$f$ as an angular Fourier series. Using this and the reversibility of~$F$ the geodesic ray transform of~$f$ can be written as a sum of generalized Abel transforms acting on the Fourier components of~$f$. By the Taylor expansions of geodesics and careful treatment of the error terms we show that these Abel transforms are injective. From this it follows that the Fourier components of~$f$ all vanish, giving the claim.

Theorem~\ref{thm:maintheorem} can be seen as a generalization of the corresponding Riemannian result in~\cite{deI:abel-transforms-x-ray-tomography} (see also~\cite{RO-herglotz-linearized}) and the proof is similar in spirit. In fact, we use the theory of Abel transforms introduced in~\cite{deI:abel-transforms-x-ray-tomography} to prove our result. However, many formulas which were explicit in~\cite{deI:abel-transforms-x-ray-tomography} become implicit and less tractable in our Finslerian case. For this reason we use the Taylor expansions of geodesics near their lowest point to show the needed regularity properties of the integral kernels of the Abel transforms (see section~\ref{sec:regularityofkernel}). 

\begin{remark}
\label{remark:familyofcurves}
We could express theorem~\ref{thm:maintheorem} in terms of a more general family of curves than geodesics. From the proof of our main theorem one sees that the curves only need to be ``sufficiently smooth" with respect to the Taylor expansions and ``sufficiently symmetric" with respect to the Finsler norm. The family of curves can be characterized by the following properties (compare to the assumptions in ~\cite{AD-finsler-scalar-plus-oneform}):
\medskip
\begin{enumerate}[label=(A\arabic*)]
\item All the curves in the family are smooth with unit speed. \label{item:smoothcurves}
\medskip
\item For every $x\in M$ and $y\in T_xM$ there is unique curve going through~$x$ to the direction~$y$. \label{item:uniquecurve}
\medskip
\item The curves depend smoothly on the initial conditions~$x$ and~$y$.
\medskip
\item Every curve reaches the boundary in finite time and has unique closest point to the origin where $\dot{r}_0=0$ and $\ddot{r}_0>0$.
\medskip
\item The curves are symmetric with respect to the lowest point and they consist of two parts where $\dot{r}>0$ and $\dot{r}<0$.
\medskip
\item The curves satisfy the weak reversibility condition~\eqref{eq:weakreversibility}. \label{item:weakreversibility}
\end{enumerate}
\medskip
The assumptions~\ref{item:smoothcurves}--\ref{item:weakreversibility} allow the existence of conjugate points on~$M$: if~$F$ is for example induced by the Riemannian metric $g=c^{-2}(r)e$ where $c=c(r)$ is smooth and satisfies the Herglotz condition and~$e$ is the Euclidean metric, then the (non-radial) geodesics of~$g$ satisfy conditions \ref{item:smoothcurves}--\ref{item:weakreversibility} (see e.g.~\cite{MO-herglotz-conjugate-points, MO-inverse-kinematic, PSU-geometric-inverse-book}). We also note that the regularity assumptions for the admissible curves could be weakened (finite number of derivatives is enough, see remark~\ref{remark:regularity}).
\end{remark}

\begin{remark}
Theorem~\ref{thm:maintheorem} can also be seen as a generalization of the famous Helgason support theorem in Euclidean geometry~\cite{HE:integral-geometry-radon-transforms} (see also~\cite[Remark 31]{deI:abel-transforms-x-ray-tomography}). According to Helgason's theorem, if a function integrates to zero on all lines not intersecting a given convex and compact set, then the function has to vanish outside that set. Since the Herglotz condition allows the presence of conjugate points (see~\cite{MO-herglotz-conjugate-points}) and on Riemannian manifolds the existence of conjugate points implies instability for the geodesic ray transform~\cite{MSU-conjugate-points-instability}, we do not expect stability for our injectivity or uniqueness result.
\end{remark}

\begin{remark}
By combining our approach with the ideas and methods of the proof of theorem~29 in~\cite{deI:abel-transforms-x-ray-tomography} we could also prove (with minor changes in the proof of theorem~\ref{thm:maintheorem}) that the attenuated geodesic ray transform is injective on (sufficiently smooth) scalar fields on our manifold~$(M, F)$ when the attenuation is a sufficiently regular radial function.
See \cite{HMS:attenuated,PSU:attenuated} for results on attenuated transforms on manifolds.
\NTR{Rephrased the statement about attenuated geodesic ray transform. This is a new a numbered remark. Added two references.}
\end{remark}

\subsection{Related results}
The geodesic ray transform has been widely studied but most of the results are obtained in the Riemannian setting. If $(M, g)$ is a compact simple Riemannian manifold with boundary (and smooth metric), then the geodesic ray transform is known to be injective~\cite{MU-reconstruction-problem-two-dimensional}.
Recently it was proved in~\cite{IK-low-regularity} that injectivity holds also when the simple Riemannian metric is only $C^{1, 1}$-regular.
Injectivity is known in the presence of conjugate points as well: if the Riemannian metric is of the form $g=c^{-2}(r)e$ where~$e$ is the Euclidean metric and the radial sound speed $c=c(r)$ satisfies the Herglotz condition, then~$\geod$ is injective on scalar fields~\cite{deI:abel-transforms-x-ray-tomography, MO-herglotz-conjugate-points, RO-herglotz-linearized, UV-local-geodesic-ray-transform} (see also~\cite{PSU-geometric-inverse-book} and the generalization to tensor fields in~\cite{SHA-herglotz-tensor}). In this case the Herglotz condition is equivalent to that the manifold has a foliation with strictly convex hypersurfaces (see also lemma~\ref{lemma:foliation}). Our main theorem is also related to the Helgason support theorem in Euclidean space~\cite{HE:integral-geometry-radon-transforms} (see~\cite[Remark 31]{deI:abel-transforms-x-ray-tomography}).

When the geodesic ray transform operates on tensor fields, the uniqueness results are known as solenoidal injectivity since one can only uniquely determine the solenoidal part of the tensor field~\cite{PSU-tensor-tomography-progress, SHA-integral-geometry-tensor-fields}. Solenoidal injectivity is known for example on two-dimensional compact simple manifolds~\cite{PSU-tensor-tomography-on-simple-surfaces}, on simply connected compact manifolds with strictly convex boundary and non-positive curvature~\cite{PS-sharp-stability-nonpositive-curvature, PS-integral-geometry-negative-curvature,SHA-integral-geometry-tensor-fields}, on certain non-compact Cartan--Hadamard manifolds~\cite{LRS-tensor-tomography-cartan-hadamard}\NTR{Removed unpublished reference.} and on manifolds which admit strictly convex foliation~\cite{dHUZ-inverting-higher-rank-tensors, SHA-herglotz-tensor, SUV-inverting-local-tensors, UV-local-geodesic-ray-transform}.
A more comprehensive treatment of the geodesic ray transform on Riemannian manifolds can be found in~\cite{IM:integral-geometry-review, PSU-tensor-tomography-progress, SHA-integral-geometry-tensor-fields}.

There are some injectivity results in the Finslerian case. It is known that the geodesic ray transform is injective on scalar fields on simple Finsler manifolds~\cite{IVA-finsler-monotonicity, SHA-finsler-injectivity}.
The geodesic ray transform is also injective on a certain family of curves on general Finsler surfaces~\cite{AD-finsler-scalar-plus-oneform}. This result extends to one-forms as well when uniqueness is understood modulo potential fields. Compared to the results in~\cite{AD-finsler-scalar-plus-oneform, IVA-finsler-monotonicity, SHA-finsler-injectivity} our theorem allows the existence of conjugate points (see e.g.~\cite{MO-herglotz-conjugate-points}). 
We also note that we could express our main theorem in terms of a family of general geodesic-like curves satisfying certain assumptions (see remark~\ref{remark:familyofcurves} and compare to the assumptions in~\cite{AD-finsler-scalar-plus-oneform}). Other injectivity results for a general family of curves can be found in~\cite{BSU-generic-family-curves, HS-weighted-doppler, MU-inverse-kinematic, SHA-ray-transform-riemannian-manifold}.

The geodesic ray transform arises naturally in the linearization of the travel time tomography or the boundary rigidity problem where one wants to uniquely determine (up to a gauge) the Riemannian metric (more generally a Finsler norm) from the distances between boundary points~\cite{SHA-integral-geometry-tensor-fields}. When we have conformal variations then the linearized problem reduces to the injectivity of the geodesic ray transform in the background geometry (see section~\ref{sec:conformal}). The travel time tomography problem was solved over a century ago for radial sound speeds satisfying the Herglotz condition~\cite{HE-inverse-kinematic-problem, WZ-kinematic-problem} (see also~\cite{PSU-geometric-inverse-book}). 
In this case the solution of the problem reduces to the inversion of an Abel transform~\cite{NOW-herglotz-abel-transform, SHE-introduction-to-seismology}.
There are also recent spectral rigidity results for radial sound speeds which satisfy the Herglotz condition~\cite{deIK-spectral-rigidity-herglotz}. 

In the more general setting boundary rigidity is known for two-dimensional compact simple Riemannian surfaces~\cite{PU-simple-manifolds-boundary-rigidity}, for manifolds admitting strictly convex foliation~\cite{SUV-boundary-rigidity-foliation, SUVZ-travel-time-tomography} and for compact simple Riemannian manifolds which are in the same conformal class~\cite{CRO-rigidity-conformal, MU-reconstruction-problem-two-dimensional, SUVZ-travel-time-tomography}. There are some Finslerian results as well including Randers metrics~\cite{MO-randers}, reversible Finsler norms which satisfy a strictly convex foliation~\cite{deILS-broken-scattering-rigidity} and projectively flat Finsler norms in the plane~\cite{ALE-convex-rigidity-plane, AM-pseudo-metrics-on-the-plane, KO-boundary-rigidity-projective-metrics}. Our main result can be seen as a boundary rigidity result up to first order for a conformal family of spherically symmetric reversible Finsler norms satisfying the Herglotz condition (see section~\ref{sec:conformal}).
A survey of the boundary rigidity or the travel time tomography problem can be found in~\cite{SUVZ-travel-time-tomography}.

\subsection{Organization of the paper}
\NTR{Removed a sentence.} In section~\ref{sec:preliminaries} we go through basic definitions and properties of Finsler manifolds and Abel transforms and we study the Herglotz condition. We prove our main theorem in section~\ref{sec:proofofmainresult}. In section~\ref{sec:regularityofkernel} we prove the regularity properties of the integral kernel of the Abel transforms. Finally, in section~\ref{sec:linearizedelastic} we discuss the linearization of the boundary rigidity problem on Finsler manifolds and the application of our result to linearized elastic travel time tomography.

\subsection*{Acknowledgements}
J.I. was supported by Academy of Finland (grants 332890, 336254, 351665, 351656).
K.M. was supported by Academy of Finland (Centre of Excellence in Inverse Modelling and Imaging, grant numbers 284715 and 309963).
We wish to thank the anonymous referee for feedback.\NTR{Thank you for your help!}

\section{Preliminaries}
\label{sec:preliminaries}
In this section we go through definitions, notation and lemmas which are needed in the proof of our main theorem. The basic theory of Finsler geometry can be found in~\cite{AL-global-aspects-finsler-geometry, BCS-introduction-finsler-geometry,  CS-riemann-finsler-geometry, SHE-lectures-on-finsler-geometry} and the geodesic ray transform is treated in detail in~\cite{SHA-integral-geometry-tensor-fields}. Generalized Abel transforms are studied for example in~\cite{deI:abel-transforms-x-ray-tomography, IL-broken-ray-abel}. 

\subsection{Finsler manifolds}
\label{sec:finslermanifolds}
Let~$M$ be a smooth manifold with or without a boundary. We use $x\in M$ to denote the base point and $y\in T_xM$ to denote the direction in the tangent space. A non-negative function $F\colon TM\to [0, \infty)$ of the tangent bundle is called a Finsler norm if it satisfies the following conditions:
\medskip
\begin{enumerate}[label=(\roman*)]
    \item\label{item:smoothness} $F$ is smooth in $TM\setminus\{0\}$
    \medskip
    \item $F(x, y)=0$ if and only if $y=0$ 
    \medskip
    \item $F(x, \lambda y)=\lambda F(x, y)$ for every $\lambda\geq 0$ 
    \medskip
    \item\label{item:convexity} $\frac{1}{2}\frac{\partial^2 F^2(x, y)}{\partial y^i\partial y^j}$ is positive definite for all $y\neq 0$. 
\end{enumerate}
\medskip
The pair $(M, F)$ is called a Finsler manifold.
In other words, the map $y\mapsto F(x, y)$ defines a Minkowski norm in~$T_xM$ for every $x\in M$.
The length of a piecewise smooth curve $\gamma\colon [a, b]\to M$ is defined as $L(\gamma)=\int_a^b F(\gamma(t),\dot{\gamma}(t))\der t$.
In this way a Finsler norm~$F$ defines a (not necessarily symmetric) distance function on~$M$.

Finsler norm~$F$ is reversible, if $F(x, -y)=F(x, y)$ for all $x\in M$ and $y\in T_xM$. Riemannian metrics are a special case of reversible Finsler norms: if $g$ is a Riemannian metric, then $F_g(x, y)=\sqrt{g_{ij}(x)y^iy^j}$ defines a reversible Finsler norm where we have used the Einstein summation convention under the square root. A distance function induced by a reversible Finsler norm is symmetric. Not all Finsler norms are reversible: examples include Randers metrics $F=F_g+\beta$ where~$g$ is a Riemannian metric and~$\beta$ is a one-form. On the other hand, there are reversible Finsler norms which are not induced by any Riemannian metric.

Using convexity property~\ref{item:convexity} we can define the Finslerian metric tensor
\begin{equation}
g_{ij}(x, y)=\frac{1}{2}\frac{\partial^2 F^2(x, y)}{\partial y^i\partial y^j}.
\end{equation}
If $F=F_g$ is induced by a Riemannian metric, then $g_{ij} (x, y)=g_{ij}(x)$ is independent of $y\in T_xM$. Using the Finslerian metric tensor one can define the Legendre transformation $L\colon TM\to T^*M$ which in the Riemannian case corresponds to the musical isomorphisms. Legendre transformation allows us to define the co-Finsler norm (or dual norm) $F^*\colon T^*M\to[0, \infty)$ so that for every $\omega\in T^*_xM$ we have
\begin{equation}
F^*(x, \omega)=\sup_{\substack{y\in T_xM \\F(x, y)=1}}\omega(y).
\end{equation}

Let $\gamma\colon [a, b]\to M$ be a smooth curve on~$M$. We call~$\gamma$ a geodesic if it is a critical point of the length functional $\gamma\mapsto L(\gamma)$. Equivalently, we say that~$\gamma$ is geodesic if it satisfies the geodesic equation
\begin{equation}
\label{eq:geodesicequation}
\ddot{\gamma}^i(t)+2G^i(\gamma(t), \dot{\gamma}(t))=0
\end{equation}
where $G^i=G^i(x, y)$ are the spray coefficients defined as
\begin{equation}
G^i(x, y)=\frac{1}{4}g^{il}(x, y)\bigg(y^k\frac{\partial^2 F^2(x, y)}{\partial x^k\partial y^l}-\frac{\partial F^2(x, y)}{\partial x^l}\bigg).
\end{equation}
Here $g^{ij}(x, y)$ is the inverse matrix of $g_{ij}(x, y)$ and we have used the Einstein summation convention. Geodesics correspond to straightest possible paths on a Finsler manifold and they minimize distances locally. It follows that if~$F$ is a reversible Finsler norm and~$\gamma$ is a geodesic of~$F$, then the reversed reparametrization~$\overleftarrow{\gamma}(t)=\gamma(-t)$ is also a geodesic of~$F$.

\subsection{The Herglotz condition}
\label{sec:herglotz}
Let~$M\subset\R^2$ be an annulus centered at the origin and let $(x^1, x^2)=(r, \theta)$ be the usual polar coordinates on~$M$. The coordinate vector fields~$\partial_r$ and~$\partial_\theta$ form a basis in every tangent space~$T_{(r, \theta)}M$. The coordinates of a tangent vector $y\in T_{(r, \theta)}M$ in this basis are denoted by~$(y^1, y^2)=(\rho, \phi)$ (see figure~\ref{fig:coordinates}). The components of~$y$ can be calculated using the differentials $\rho=\der r(y)$ and $\phi=\der\theta(y)$. Hence we can identify $\rho\leftrightarrow\der r$ and $\phi\leftrightarrow\der\theta$. 
 
We say that~$F$ is a spherically symmetric Finsler norm on~$M$ if $U^*F=F$ for every $U\in SO(2)$ where the pullback of a Finsler norm via smooth map~$\Phi$ is defined as $(\Phi^* F)(x, y)=F(\Phi (x), \der\Phi_x (y))$. Spherical symmetry implies that $F=F(r, \rho, \phi)$ is independent of~$\theta$ and the angular momentum $L(r, \rho, \phi)=\frac{1}{2}\partial_{\phi}F^2(r, \rho, \phi)$ is conserved along geodesics.
Further, we say that a spherically symmetric reversible Finsler norm~$F$ on~$M$ satisfies the Herglotz condition, if
\begin{equation}
\label{eq:herglotzcond}
\partial_r F^2(r, 0, \phi)>0
\end{equation}
for all $r\in (R, 1]$ and $\phi\neq 0$ where $R\in (0, 1)$. 

If $\gamma$ is a geodesic on $(M, F)$, we write it in polar coordinates as $\gamma(t)=(r(t), \theta(t))$.\NTR{Added the expression of $\gamma$ in polar coordinates so that the next three formulas can be understood.}
The geodesic equation for the radial coordinate becomes
\begin{equation}
\label{eq:geodesicradial}
\ddot{r}(t)=-\frac{g^{1l}(\gamma(t), \dot{\gamma}(t))}{2}\big(\dot{\gamma}^k(t)\partial_{x^k}\partial_{y^l}F^2(\gamma(t), \dot{\gamma}(t))-\partial_{x^l}F^2(\gamma(t), \dot{\gamma}(t))\big).
\end{equation}
If $\dot{r}(t_0)=0$, then by spherical symmetry $\partial_\theta F^2(x, y)=0$ and
\begin{equation}
\ddot{r}(t_0)=\frac{1}{2}g^{11}(\gamma(t_0), \dot{\gamma}(t_0))\partial_{r}F^2(\gamma(t_0), \dot{\gamma}(t_0)).
\end{equation}
Since the Finslerian metric tensor $g_{ij}(x, y)$ is positive definite one sees that the Herglotz condition~\eqref{eq:herglotzcond} is equivalent to that
\begin{equation}
\label{eq:herglotzradialderivative}
\text{if $\dot{r}(t_0)=0$, then $\ddot{r}(t_0)>0$}.
\end{equation}
Equation~\eqref{eq:herglotzradialderivative} means that geodesics which are initially tangential to circles curve outwards. 
One can also see from the geodesic equations that the Herglotz condition forbids circles being geodesics.

An example of a spherically symmetric reversible Finsler norm satisfying~\eqref{eq:herglotzcond} is the Finsler norm (see also~\cite{YN-finsler-seismic-ray-path})
\begin{equation}
\label{eq:finslernormexample}
F^2(r, \rho, \phi)=\frac{\rho^2+r^2\phi^2}{c^2(r, \rho, \phi)}
\end{equation}
where the (anisotropic) sound speed $c=c(r, \rho, \phi)$ is reversible $c(r, -\rho, -\phi)=c(r, \rho, \phi)$ and satisfies the Herglotz condition
\begin{equation}
\label{eq:herglotzanisotropicspeed}
\frac{\partial}{\partial r}\bigg(\frac{r}{c(r, 0, \phi)}\bigg)>0
\end{equation}
for all $r\in (R, 1]$ and $\phi\neq 0$. 
Note that if $c=c(r)$ is radial (i.e.~$F$ is Riemannian), then this reduces to the usual Herglotz condition 
$$
\frac{\der}{\der r}\bigg(\frac{r}{c(r)}\bigg)>0
$$
for all $r\in (R, 1]$ and the Finsler norm defined by equation~\eqref{eq:finslernormexample} is a natural generalization of the Riemannian metric
$$
g=\frac{\der r^2+r^2\der\theta^2}{c^2(r)}.
$$

We say that a Finsler manifold $(M, F)$ with boundary has a strictly convex foliation, if there is a smooth function $\psi\colon M\to\R$ such that
\medskip
\begin{enumerate}[label=(\alph*)]
    \item $\psi^{-1}\{0\}=\partial M$, $\psi^{-1}(0, S]=\text{int}(M)$ and $\psi^{-1}(S)$ has empty interior.\label{item:foliation1}
    \medskip
    \item For each $s\in [0, S)$ the set $\Sigma_s=\psi^{-1}(s)$ is a strictly convex smooth surface in the sense that $\der \psi\neq 0$ and any geodesic~$\gamma$ having initial conditions in $T\Sigma_s$ satisfies $\frac{\der^2}{\der t^2}\psi(\gamma(t))|_{t=0}<0$.\label{item:foliation2}
\end{enumerate}
\medskip
For more details and discussion see~\cite{deILS-broken-scattering-rigidity}.
The Herglotz condition~\eqref{eq:herglotzcond} implies that~$M$ has strictly convex foliation, i.e. circles $\aabs{x}=r$ are strictly convex.

\begin{lemma}
\label{lemma:foliation}
Let $M=\bar{B}(0, 1)\setminus\bar{B}(0, R)\subset\R^2$ where $R\in (0, 1)$ and equip~$M$ with a spherically symmetric reversible Finsler norm~$F$ which satisfies the Herglotz condition~\eqref{eq:herglotzcond}. Then $(M, F)$ admits a strictly convex foliation. 
\end{lemma}

\begin{proof}
Define the function $\psi(x)=1-\aabs{x}^2$. Then~$\psi$ is smooth, $\der\psi(x)=-2x\neq 0$ and it is easy to check the requirements in~\ref{item:foliation1}. The level sets of~$\psi$ are circles. Let~$\gamma$ be a geodesic which is initially tangential to a circle. We need to check that $\frac{\der^2}{\der t^2}\psi(\gamma(t))|_{t=0}<0$.\NTR{Added these three sentences for clarity.} We can calculate $\frac{\der}{\der t}\psi(\gamma(t))=-2\gamma(t)\cdot\dot{\gamma}(t)$ and hence $\frac{\der^2}{\der t^2}\psi(\gamma(t))|_{t=0}=-2(\dot{\gamma}(0)\cdot\dot{\gamma}(0)+\gamma(0)\cdot\ddot{\gamma}(0))$.\NTR{Added computation for the first derivative.} By spherical symmetry we can assume without loss of generality\NTR{Added "without loss of generality".} that $\theta(0)=0$. Since $\dot{r}(0)=0$ we have $\ddot{r}(0)>0$ due to the Herglotz condition~\eqref{eq:herglotzradialderivative}. Therefore $\gamma(0)\cdot\ddot{\gamma}(0)=r(0)\ddot{r}(0)>0$ which implies that $\frac{\der^2}{\der t^2}\psi(\gamma(t))|_{t=0}<0$. This proves that $(M, F)$ has a strictly convex foliation.\NTR{The end of the proof has been written more clearly.}
\end{proof}

\begin{lemma}
\label{lemma:non-trapping}
Let $(M, F)$ be as in lemma~\ref{lemma:foliation}.
If~$\gamma$ is a geodesic such that $\dot{r}(0)=0$, then~$\gamma$ reaches the boundary~$\partial M$ in finite time from both ends and~$\gamma$ consists of two symmetric parts (with respect to $(r(0), \theta(0))$) where $\dot{r}<0$ and $\dot{r}>0$.
\end{lemma}

\begin{proof}
Let~$\gamma$ be a geodesic such that $\dot{r}(0)=0$. By spherical symmetry we can assume without loss of generality\NTR{Added "without loss of generality".} that $\theta(0)=0$. Let $\psi(x)=1-\aabs{x}^2$ be as in the proof of lemma~\ref{lemma:foliation}.\NTR{Added the definition of the function $\psi$.} According to~\cite[Lemma 14]{deILS-broken-scattering-rigidity}, if~$\eta$ is a geodesic such that $\frac{\der}{\der t}\psi(\eta(t))|_{t=0}<0$, then~$\eta$ reaches the boundary in finite time as~$t$ increases. Now $\frac{\der}{\der t}\psi(\gamma(t))|_{t=0}=-2\gamma(0)\cdot\dot{\gamma}(0)=0$ so we can not directly use~\cite[Lemma 14]{deILS-broken-scattering-rigidity}. The Herglotz condition implies that $\ddot{r}(0)>0$ so there is $\epsilon_1>0$ such that $\dot{r}(t)>0$ for all $t\in (0, \epsilon_1)$. Since~$\gamma$ has unit speed and~$F$ is homogeneous we have $1=F(r(0), 0, 0, \dot{\theta}(0))=\abs{\dot{\theta}(0)}F(r(0), 0, 0, 1)$ so $\abs{\dot{\theta}(0)}\neq 0$, which implies that the angular variable $\theta(t)$ is either increasing or decreasing at $t=0$. In either case there is $\epsilon_2>0$ such that $\theta(t)\dot{\theta}(t)>0$ for all $t\in (0, \epsilon_2)$. These observations imply that there is $\epsilon>0$ such that $\gamma(\epsilon)\cdot\dot{\gamma}(\epsilon)=r(\epsilon)\dot{r}(\epsilon)+r^2(\epsilon)\theta(\epsilon)\dot{\theta}(\epsilon)>0$.\NTR{Added more details why such $\epsilon>0$ exists.} Now defining $\eta(t)=\gamma(t+\epsilon)$ we obtain that~$\frac{\der}{\der t}\psi(\eta(t))|_{t=0}<0$ so~$\eta$ and hence~$\gamma$ reaches the boundary in finite time as~$t$ increases. Using similar reasoning for the reversed geodesic $\overleftarrow{\gamma}(t)=\gamma(-t)$ we obtain that~$\gamma$ has finite length and reaches the boundary from its both ends. The symmetry of~$\gamma$ with respect to $(r(0), \theta(0))$ follows from the reversibility and spherical symmetry of~$F$. The Herglotz condition in turn implies that~$\dot{r}$ cannot have more than one zero since all critical points of~$r(t)$ have to be local minima. Hence~$\gamma$ consists of a rising part where $\dot{r}>0$ and a descending part where $\dot{r}<0$.
\end{proof}

Lemma~\ref{lemma:non-trapping} implies that those geodesics which are initially tangential to circles have finite length. This is enough for us since we only use this type of geodesics in the proof of theorem~\ref{thm:maintheorem}. We note that in the Riemannian case (i.e. for the metric $g=c^{-2}(r)e$) the Herglotz condition implies that the whole manifold is non-trapping (see e.g.~\cite{MO-inverse-kinematic, PSU-geometric-inverse-book}). We do not know if this is true also in our Finslerian setting. 

The Herglotz condition also allows us to make a change of coordinates on geodesics which are initially tangential to circles. We can interchange between the time parameter~$t$ and the radial coordinate~$r$ since by the Herglotz condition we have $\partial r/\partial t>0$ on the rising part and $\partial r/\partial t<0$ on the descending part of geodesics (see lemma~\ref{lemma:non-trapping}). Here we have a partial derivative since generally~$r$ also depends on the lowest point $(r_0, \theta_0)$ of the geodesic. In particular, we can change between the coordinates $(r_0, t)$ and $(r_0, r)$. This coordinate transformation is treated in more detail in sections~\ref{sec:taylorexpansions} and~\ref{sec:coordinatechange}.\NTR{Added this paragraph on change of coordinates.}

\subsection{Geodesic ray transform and Abel transforms}
\label{sec:geodesicraytransform}
Throughout this section we assume that $M\subset\R^2$ is an annulus centered at the origin equipped with a spherically symmetric reversible Finsler norm~$F$ satisfying the Herglotz condition~\eqref{eq:herglotzcond}.

The geodesic ray transform of a scalar field $f\colon M\to\R$ is defined as
$$
\geod f(\gamma)=\int_\gamma f\der s
$$
where~$\gamma$ is a unit speed geodesic. The integrals are finite for sufficiently regular functions when the geodesic~$\gamma$ has finite length (e.g. if~$\gamma$ has unique lowest point to the origin). 

We use the following angular Fourier series expansion which allows us to write the geodesic ray transform of~$f$ in terms of the geodesic ray transforms of the component functions $f_k(r, \theta)=a_k(r)e^{ik\theta}$.

\begin{lemma}[{\cite[Lemma 20]{deI:abel-transforms-x-ray-tomography}}]
\label{lemma:angularfourier}
If $f\in L^2(M)$, then it can be written as an angular Fourier series
\begin{equation}
\label{eq:angularfourier}
f(r, \theta)=\sum_{k\in\Z} a_k(r)e^{ik\theta}
\end{equation}
where $a_k\in L^2([R, 1])$ and the series convergences to~$f$ with respect to the $L^2(M)$-norm.
\end{lemma}

Let $\gamma\colon [-T, T]\rightarrow M$, $\gamma(t)=(r(t), \theta_0+\omega(t))$, be a geodesic with lowest point $(r_0, \theta_0)$ and highest point at $r=1$. Since~$F$ is reversible~$\gamma$ is symmetric with respect to $(r_0, \theta_0)$ (see lemma~\ref{lemma:non-trapping}), i.e. $(r(-t), \theta(-t))=(r(t), \theta_0-\theta(t))$. Using this symmetry and change of variables $t\to r$ (which is also possible due to lemma~\ref{lemma:non-trapping}) we obtain the following formula for the Fourier components
\begin{align}
\geod f_k(r_0, \theta_0)&=\int_{-T}^T a_k(r(t))e^{ik(\theta_0+\omega(t))}\der t=2e^{ik\theta_0}\int_0^T a_k(r(t))\cos(k\omega(t))\der t \nonumber \\
&= 2e^{ik\theta_0}\int_{r_0}^1 (r-r_0)^{-1/2}\widetilde{K}_k(r_0, r)a_k(r)\der r=2e^{ik\theta_0}\abel_ka_k(r_0) \label{eq:geodesictransformfouriercomponent}
\end{align}
where $\widetilde{K}_k(r_0, r)=K(r_0, r)\cos(k\omega(r_0, r))$, $K(r_0, r)=(r-r_0)^{1/2}(\dot{r}(r_0, r))^{-1}$ and
\begin{equation}
\label{eq:abel}
\abel_k h(x)=\int_x^1(y-x)^{-1/2}\widetilde{K}_k(x, y)h(y)\der y.
\end{equation}
We call the integral transform~$\abel_k$ in~\eqref{eq:abel} a generalized Abel transform.

The Abel transforms~$\abel_k$ are a special case of the more general integral transforms
\begin{equation}
\label{eq:generalintegraltransform}
I_{\mathcal{K}}^\alpha h(x)=\int_x^1 (y-x)^{-\alpha}\mathcal{K}(x, y)h(y)\der y
\end{equation}
where $\mathcal{K}\colon\Delta\to\R$ is any bounded function, $\alpha\in [0, 1)$, $\Delta=\{(u_1, u_2)\in\R^2: 0\leq u_1\leq u_2\leq 1\}$ and $h\colon [0, 1]\to \R$ is regular enough so that the integral in~\eqref{eq:generalintegraltransform} is well-defined. These type of integral transforms were studied in~\cite{deI:abel-transforms-x-ray-tomography} and they satisfy the following important properties.

\begin{lemma}[{\cite[Theorem 4]{deI:abel-transforms-x-ray-tomography}}]
\label{lemma:generaltransformcontinuity}
The transform $I_{\mathcal{K}}^\alpha\colon L^p([0, 1])\to L^q([0, 1])$ is well-defined and continuous when $\alpha+1/p<1+1/q$. In particular, this holds when $p>1/(1-\alpha)$, $q<1/\alpha$ or $p=q$. The norm of this mapping satisfies $\aabs{I_{\mathcal{K}}^\alpha}_{L^p\to L^q}=\bigoh(\sup_\Delta\abs{\mathcal{K}})$.
\end{lemma}

\begin{lemma}[{\cite[Theorem 12]{deI:abel-transforms-x-ray-tomography}}]
\label{lemma:generaltransforminjective}
Let $\alpha\in [0, 1)$. Suppose $\mathcal{K}\colon\Delta\to\R$ is bounded everywhere, non-zero on the diagonal $\{(u, u)\in\R^2: 0\leq u\leq 1\}$ and Lipschitz continuous in some neighborhood of the diagonal. If $h\in L^1([0, 1])$ satisfies $I_{\mathcal{K}}^\alpha h(x)=0$ for almost all $x\geq r$ for some $r\in [0, 1)$, then $h(x)=0$ for almost all $x\geq r$. In particular, $I_{\mathcal{K}}^\alpha\colon L^1([0, 1])\to L^1([0, 1])$ is injective.
\end{lemma}

The above lemmas hold also if we replace~$\Delta$ with $\Delta_R=\{(u_1, u_2)\in\R^2: R\leq u_1\leq u_2\leq 1\}$ (see~\cite{deI:abel-transforms-x-ray-tomography} for details). We show in sections~\ref{sec:proofofmainresult} and~\ref{sec:regularityofkernel} that the Abel transforms~$\abel_k$ defined by equation~\eqref{eq:abel} satisfy the assumptions in lemmas~\ref{lemma:generaltransformcontinuity} and~\ref{lemma:generaltransforminjective}.

\section{Proof of the main theorem}
\label{sec:proofofmainresult}
In this section we prove our main theorem.
The idea of the proof is the following. Using angular Fourier series the injectivity problem can be reduced to the invertibility problem of generalized Abel transforms acting on the Fourier components of~$f$. Writing the Taylor expansions of geodesics and analyzing the error terms we show in section~\ref{sec:regularityofkernel} that the integral kernel $\widetilde{K}_k=\widetilde{K}_k(r_0, r)$ satisfies the regularity properties which are needed in lemmas~\ref{lemma:generaltransformcontinuity} and~\ref{lemma:generaltransforminjective}. From this it follows that the Abel transforms~$\abel_k$ are injective, implying that the Fourier components of~$f$ all have to vanish, which proves the claim.

The following two lemmas form the core of our proof since they imply that the integral kernel~$\widetilde{K}_k$ is locally regular enough so that we can use the theory of Abel transforms developed in~\cite{deI:abel-transforms-x-ray-tomography} (i.e. lemmas~\ref{lemma:generaltransformcontinuity} and~\ref{lemma:generaltransforminjective}).
\begin{lemma}
\label{lemma:abelkernel}
Let $R\in (0, 1)$ and $\Delta_R=\{(u_1, u_2)\in\R^2: R\leq u_1\leq u_2\leq 1\}$.
Define the integral kernel $K\colon\Delta_R\rightarrow\mathbb{R}$ as
$$
K(r_0, r)=(r-r_0)^{1/2}(\dot{r}(r_0, r))^{-1}
$$
where $\dot{r}=\dot{r}(r_0, t)$ is obtained from the solution $\gamma(r_0, t)=(r(r_0, t), \theta(r_0, t))$ of the geodesic equation with initial conditions $r(0)=r_0$ and $\dot{r}(0)=0$ and we have changed the parameter $t\to r$ using lemma~\ref{lemma:non-trapping}.
Then $K=K(r_0, r)$ is bounded everywhere in~$\Delta_R$, non-zero on the diagonal $\{(u, u): R\leq u\leq 1\}$ and Lipschitz continuous in a small neighborhood of the diagonal.
\end{lemma} 

\begin{lemma}
\label{lemma:angularchange}
Let $\omega(r_ 0, t)=\theta(r_0, t)-\theta_0$ be the angular change of a geodesic with lowest point $(r_0, \theta_0)$ and $R\in (0, 1)$. The map $(r_0, r)\mapsto \omega^2(r_0, r)$ is Lipschitz continuous in a small neighborhood of the diagonal $\{(u, u): R\leq u\leq 1\}$.
\end{lemma}

The above lemmas are quite technical and laborious to prove. Therefore we have devoted our\NTR{Added "our".} own section to the proofs (see section~\ref{sec:regularityofkernel}). Lemmas~\ref{lemma:abelkernel} and~\ref{lemma:angularchange} imply the following important result for the Abel transforms~$\abel_k$.

\begin{lemma}
\label{lemma:abelinjective}
Let~$\abel_k$ be the Abel transforms defined by equation~\eqref{eq:abel}. Then $\abel_k\colon L^2([R, 1])\rightarrow L^2([R, 1])$ are equicontinuous. Furthermore, the transforms $\abel_k\colon L^1([R, 1])\to L^1([R, 1])$ are injective.
\end{lemma}

\begin{proof}
By lemma~\ref{lemma:abelkernel} the kernel~$K=K(r_0, r)$ is bounded. This implies that $\widetilde{K}_k(r_0, r)=K(r_0, r)\cos(k\omega(r_0, r))$ is bounded too and we can use lemma~\ref{lemma:generaltransformcontinuity} to deduce that~$\abel_k$ are equicontinuous on~$L^2([R, 1])$.
From lemmas~\ref{lemma:abelkernel} and~\ref{lemma:angularchange} we get that~$K=K(r_0, r)$ and~$(r_0, r)\mapsto \omega^2(r_0, r)$ are Lipschitz in a small neighborhood of the diagonal of~$\Delta_R$ and~$K$ is non-zero on the diagonal.
The Lipschitz continuity of $(r_0, r)\mapsto\cos(k\omega(r_0, r))$ in a neighborhood of the diagonal follows from the fact that $\cos(z)=h(z^2)$ where~$h$ is an analytic function formed from the Taylor series of cosine by replacing the powers~$z^{2n}$ with~$z^n$.
Therefore~$\widetilde{K}_k$ is bounded, non-zero on the diagonal and Lipschitz in a small neighborhood of the diagonal. We can use lemma~\ref{lemma:generaltransforminjective} to obtain that~$\abel_k$ are injective on $L^1([R, 1])$.
\end{proof}

Now we are ready to prove our main theorem. The proof is short since it relies on many auxiliary (and technical) lemmas.   

\begin{proof}[Proof of theorem~\ref{thm:maintheorem}]
As was mentioned in remark~\ref{remark:reduction} it is enough to prove the claim for $n=2$.
We have to show that if $\int_\gamma f\der s=0$ for all geodesics~$\gamma$, then $f=0$.
Using lemma~\ref{lemma:angularfourier} we write $f$ as angular Fourier series 
\begin{equation}
f(r, \theta)=\sum_{k\in\Z}a_k(r)e^{ik\theta}=\sum_{k\in\Z}f_k(r, \theta)
\end{equation}
where $a_k\in L^2([R, 1])$ and convergence is in $L^2(M)$.
We consider geodesics~$\gamma$ which have unique lowest point and hence finite length due to lemma~\ref{lemma:non-trapping}. 
By lemma~\ref{lemma:abelinjective} the Abel transforms~$\abel_k$ are equicontinuous which implies that~$\geod$ is continuous on~$L^2(M)$. Hence the geodesic ray transform $\geod f$ can be calculated termwise and we obtain (see equation~\eqref{eq:geodesictransformfouriercomponent})
\begin{equation}
\geod f(r, \theta)=\sum_{k\in\Z}\geod f_k(r, \theta)=2\sum_{k\in\Z}\abel_k a_k(r)e^{ik\theta}
\end{equation}
where we have parameterized geodesics with their closest point $(r, \theta)$ to the origin. 
Since $\geod f(\gamma)=0$ for all geodesics~$\gamma$ we have that $\geod f(r, \theta)=0$ for all $r\in (R, 1]$. This implies that $\abel_k a_k(r)=0$ for all $r\in (R, 1]$. Using injectivity of~$\abel_k$ (lemma~\ref{lemma:abelinjective}) we see that $a_k(r)=0$ for all $r\in (R, 1]$. Therefore $f=0$, giving the claim.
\end{proof}

\section{Regularity of the integral kernel}
\label{sec:regularityofkernel}
In this section we prove lemmas~\ref{lemma:abelkernel} and~\ref{lemma:angularchange}. The quite technical proofs are based on careful treatment of the Taylor expansions of component functions of geodesics.

\subsection{Taylor expansions and weak reversibility}
\label{sec:taylorexpansions}
Let $(r_0, \theta_0)\in M$. As~$F$ is reversible there is unique geodesic~$\gamma(r_0, t)=(r(r_0, t), \theta(r_0, t))$ modulo orientation so that $(r_0, \theta_0)$ is its lowest point. Here we have explicitly written down the dependence on~$r_0$ since we need to know the regularity of the integral kernel $\widetilde{K}_k=\widetilde{K}_k(r_0, r)$ with respect to~$r_0$. 
Reversibility of~$F$ also implies that $r(r_0, t)=r(r_0, -t)$ and $\theta(r_0, t)=\theta_0-\theta(r_0, -t)$ (see lemma~\ref{lemma:non-trapping}). Differentiating this with respect to~$t$ we obtain
\begin{equation}
\label{eq:weakreversibility}
\begin{cases}
\dddot{r}(r_0, 0)&=0 \\
\ddot{\theta}(r_0, 0)&=0.
\end{cases}
\end{equation}
Condition~\eqref{eq:weakreversibility} can be called ``weak reversibility" of a Finsler norm since it does not necessarily require reversibility. 

Using weak reversibility we write the Taylor expansion for the second derivative of the radial coordinate
$$
\ddot{r}(r_0, t)=a(r_0)+E_1(r_0, t), \quad E_1(r_0, t)=\bigoh(t^2),
$$
where $a(r_0)=\ddot{r}(r_0, 0)>0$ for all $r_0\in (R, 1)$ by the Herglotz condition.
Since $\dot{r}(r_0, 0)=0$ we can integrate the expansion for $\ddot{r}(r_0, t)$ to obtain that
\begin{align*}
\dot{r}(r_0, t)&=a(r_0)t+E_2(r_0, t) \quad\text{with}\quad E_2(r_0, t)=\bigoh(t^3) \text{ and} \\
r(r_0, t)-r_0&=\frac{a(r_0)t^2}{2}+E_3(r_0, t) \quad\text{with}\quad E_3(r_0, t)=\bigoh(t^4).
\end{align*}
Here $f(t)=\bigoh(h(t))$ means that $\abs{f(t)}\leq M\abs{h(t)}$ for small~$t$ where $M>0$ is constant.
Note that the maps $(r_0, t)\mapsto r(r_0, t)$ and $(r_0, t)\mapsto \theta(r_0, t)$ are smooth because of smooth dependence on initial conditions. Therefore $a=a(r_0)$ is bounded both from above and below by a positive constant.

\subsection{Changing between time and radial coordinate}
\label{sec:coordinatechange}
From the Taylor expansion of $r=r(r_0, t)$ we get an important relation $t^2\approx r-r_0$ for small~$t$, i.e. there is constant $C>0$ such that
$$
\frac{1}{C}t^2\leq r-r_0 \leq Ct^2
$$
when $t$ (or equivalently $r-r_0$) is small enough. Therefore when we write estimates using the Taylor expansions it does not matter whether we express them in terms of~$t$ or $r-r_0$.

When we change between~$t$ and~$r$ we need to know how the derivatives transform. We introduce the coordinates
\begin{align*}
z=(x, t)&=(r_0, t) \\
\tilde{z}=(y, r)&=(r_0, r).
\end{align*}
Coordinate transformation $z\rightarrow\tilde{z}$ is well-defined because the Herglotz condition implies that $\partial r/\partial t>0$ on the rising part of the geodesic and $\partial r/\partial t<0$ on the descending part (see lemma~\ref{lemma:non-trapping}). Using the chain rule we see that
\begin{align*}
\frac{\partial}{\partial y}&=\frac{\partial}{\partial\tilde{z}^1}=\frac{\partial}{\partial x}+\frac{\partial t}{\partial y}\frac{\partial}{\partial t} \\
\frac{\partial}{\partial r}&=\frac{\partial}{\partial\tilde{z}^2}=\frac{\partial t}{\partial r}\frac{\partial}{\partial t}.
\end{align*}
Notice that $\partial t/\partial r=(\partial r/\partial t)^{-1}$ which can be seen for example by looking at the Jacobian matrices of the transformations $z=\Psi(\tilde{z})$ and $\tilde{z}=\Psi^{-1}(z)$.

\subsection{Proofs of the lemmas}
The strategy to prove lemmas~\ref{lemma:abelkernel} and~\ref{lemma:angularchange} is the following. We use the Taylor expansions for the coordinate functions of geodesics to calculate the derivatives with respect to the variables~$r_0$ and~$r$. By a careful treatment of the error terms we show that both of the derivatives are bounded when~$t$ (or equivalently $r-r_0$) is small. Using the mean value theorem we obtain that $K=K(r_0, r)$ and $(r_0, r)\mapsto \omega^2(r_0, r)$ are Lipschitz in a small neighborhood of the diagonal of $\Delta_R=\{(u_1, u_2)\in\R^2: R\leq u_1\leq u_2\leq 1\}$.

We start by proving lemma~\ref{lemma:abelkernel}.

\begin{proof}[Proof of lemma~\ref{lemma:abelkernel}]
We write $\dot{r}^{-1}(r_0, r):=(\dot{r}(r_0, r))^{-1}$ etc.
The leading order behaviour of the kernel~$K$ can be seen by writing the expansion
\begin{align*}
K(r_0, r)&=(r-r_0)^{1/2}\dot{r}^{-1}(r_0, r)=\bigg(\sqrt{\frac{a(r_0)}{2}}t+\bigoh(t^3)\bigg)\bigg(\frac{1}{a(r_0)t}+\bigoh(t)\bigg) \\
&=\frac{1}{\sqrt{2a(r_0)}}+\bigoh(t^2).
\end{align*}
From this we easily see that~$K$ is non-zero on the diagonal and bounded everywhere in~$\Delta_R$.

We then focus on the derivative $\partial_{r}K(r_0, r)$. Using the chain rule we get
$$
\partial_{r}K(r_0, r)=\frac{\partial}{\partial r}\big((r-r_0)^{1/2}\dot{r}^{-1}\big)
=\dot{r}^{-3}(r-r_0)^{-1/2}\bigg(\frac{1}{2}\dot{r}^2-(r-r_0)\ddot{r}\bigg).
$$
The Taylor expansions from section~\ref{sec:taylorexpansions} imply that
\begin{align*}
\frac{1}{2}\dot{r}^2&=\frac{a^2(r_0)}{2}t^2+\bigoh(t^4) \\
(r-r_0)\ddot{r}&=\frac{a^2(r_0)}{2}t^2+\bigoh(t^4) \\
\dot{r}^3&=a^3(r_0)t^3+\bigoh(t^5).
\end{align*}
From the expression for $\dot{r}^3$ we obtain
$$
\frac{1}{\dot{r}^3}=\frac{1}{a^3(r_0)t^3+\bigoh(t^3)}=\frac{1}{a^3(r_0)t^3}+\bigoh(t^{-1})=\bigoh(t^{-3}). 
$$
Thus finally we have
$$
\partial_r K(r_0, r)=\bigoh(t^{-3})\cdot \bigoh(t^{-1})\cdot\bigoh(t^4)=\bigoh(1)
$$
which implies that the derivative $\partial_r K(r_0, r)$ is
bounded when~$t$ is small, or equivalently when $r-r_0$ is small.

The estimate for the derivative $\partial_{r_0}K(r_0, r)$ is a little bit trickier. We use the coordinates $(x, t)=(r_0, t)$ and $(y, r)=(r_0, r)$ introduced in section~\ref{sec:coordinatechange}. First we obtain that
$$
\partial_y K(y, r)=\frac{\partial}{\partial y}\big((r-y)^{1/2}\dot{r}^{-1}\big)=\frac{-(r-y)^{-1/2}}{2}\dot{r}^{-1}-(r-y)^{1/2}\dot{r}^{-2}\frac{\partial \dot{r}}{\partial y}.
$$
We use the derivative transformations from section~\ref{sec:coordinatechange} to see that
$$
\frac{\partial \dot{r}}{\partial y}=\frac{\partial \dot{r}}{\partial x}+\frac{\partial t}{\partial y}\frac{\partial \dot{r}}{\partial t}=\frac{\partial \dot{r}}{\partial x}+\frac{\partial t}{\partial y}\ddot{r}
$$
where
$$
\frac{\partial \dot{r}}{\partial x}=a^{\prime}(r_0)t+\partial_{x}E_2.
$$
The Taylor expansion for~$r$ allows us to write
$$
t(y, r)=\sqrt{\frac{2}{a(y)}((r-y)-E_3(y, r))}
$$
which can be differentiated with respect to~$y$ 
\begin{equation}
\label{eq:timederivativey}
\frac{\partial t}{\partial y}=-\frac{a(y)(1+\partial_y E_3)+a^{\prime}(y)((r-y)-E_3)}{a^2(y) t}.
\end{equation}
Therefore the derivative becomes
\begin{align*}
\partial_y K(y, r)&=\frac{-(r-y)^{-1/2}}{2}\dot{r}^{-1}-(r-y)^{1/2}\dot{r}^{-2}\bigg((a^{\prime}(y)t+\partial_{x}E_2) \\&-\frac{a(y)(1+\partial_y E_3)+a^{\prime}(y)((r-y)-E_3)}{a^2(y) t}
(a(y)+E_1)\bigg).
\end{align*}
We estimate the different terms separately. 

First of all
$$
(r-y)^{1/2}\dot{r}^{-2}=\bigoh(t)\cdot\bigoh(t^{-2})=\bigoh(t^{-1}).
$$
Now $E_2(x, 0)=0$ for all $x$ and therefore $\partial_x E_2(x, 0)=0$. Compactness and\NTR{Added compactness.} smooth dependence on initial conditions give us that $\partial_x E_2$ is Lipschitz with respect to~$t$. Thus  $\partial_x E_2(x, t)=\bigoh(t)$ and
\[
(r-y)^{1/2}\dot{r}^{-2}\cdot (a^{\prime}(y)t+\partial_x E_2)=\bigoh(t^{-1})\cdot\bigoh(t)=\bigoh(1).
\]
Similarly $(r-y)=\bigoh(t^2)$ and $E_3=\bigoh(t^4)=\bigoh(t^2)$ so 
\[
(r-y)^{1/2}\dot{r}^{-2}\cdot \frac{a^{\prime}(y)((r-y)-E_3)}{a^2(y)t}=\bigoh(t^{-1})\cdot \bigoh(t)=\bigoh(1).
\]
Additionally $E_1=\bigoh(t^2)$ and
$$
(r-y)^{1/2}\dot{r}^{-2}\cdot \frac{a(y)E_1}{a^2(y)t}=\bigoh(t^{-1})\cdot \bigoh(t)=\bigoh(1).
$$

Next we take a look at the term $\partial_y E_3$. From the transformation law $\frac{\partial}{\partial y}=\frac{\partial}{\partial x}+\frac{\partial t}{\partial y}\frac{\partial}{\partial t}$ we get
\begin{align*}
\partial_y E_3&=\partial_x E_3-\partial_t E_3\frac{a(y)+a(y)\partial_y E_3+a^{\prime}(y)((r-y)-E_3)}{a^2(y)t}
\\ \Leftrightarrow \partial_y E_3&=\frac{a^2(y)t\partial_x E_3-\partial_t E_3(a(y)+a^{\prime}(y)((r-y)-E_3))}{a^2(y)t+a(y)\partial_t E_3}.
\end{align*}
Now $\partial_x E_2=\bigoh(t)$ and
\[
E_{i}(x, t)=\int_0^t E_{i-1}(x, s)\text{d}s, \quad i\in\{2, 3\}.
\]
By smoothness we have
\[
\partial_x E_3(x, t)=\int_0^t\partial_x E_2(x, s)\text{d}s
\]
so $\partial_x E_3=\bigoh(t^2)$. Also $\partial_t E_3=E_2=\bigoh(t^3)$ and we have an estimate for the derivative
\begin{align*}
\partial_y E_3=\frac{\bigoh(t^3)}{a^2(y)t+a(y)\partial_t E_3}=\bigoh(t^2)(1+\bigoh(t^2))=\bigoh(t^2).
\end{align*}
Thus 
$$
(r-y)^{1/2}\dot{r}^{-2}\cdot \frac{a(y)\partial_y E_3}{a^2(y)t}=\bigoh(t^{-1})\cdot \bigoh(t)=\bigoh(1).
$$

Finally the $y$-derivative of the integral kernel is
\begin{align*}
\partial_y K(y,r )&=\frac{-(r-y)^{-1/2}}{2}\dot{r}^{-1}+(r-y)^{1/2}\dot{r}^{-2}\frac{1}{t}+\bigoh(1)\\ &=\dot{r}^{-2}(r-y)^{-1/2}t^{-1}\bigg(-\frac{\dot{r}t}{2}+(r-y)\bigg)+\bigoh(1) \\
&=\dot{r}^{-2}(r-y)^{-1/2}t^{-1}\bigg(-\frac{a(y)t^2}{2}-\frac{tE_2}{2}+\frac{a(y)t^2}{2}+E_3\bigg)+\bigoh(1) \\ &= \bigoh(t^{-4})\cdot \bigoh(t^4)+\bigoh(1)=\bigoh(1).
\end{align*}
This implies that $\partial_{r_0}K(r_0, r)$ is bounded for small~$t$, or equivalently for small $r-r_0$.

Let $\epsilon>0$ be small enough so that both $\partial_{r}K(r_0, r)$ and $\partial_{r_0}K(r_0, r)$ are bounded when $r-r_0<\epsilon$. Now if $(r_0, r), (\widetilde{r_0}, \widetilde{r})\in\Delta_R^\epsilon=\{(u_1, u_2)\in\Delta_R: u_2-u_1<\epsilon\}$, then the mean value theorem implies that
\begin{align*}
\abs{K(r_0, r)-K(\widetilde{r_0}, \widetilde{r})}&\leq\abs{\nabla K(\widehat{r_0}, \widehat{r})}\abs{(r_0, r)-(\widetilde{r_0}, \widetilde{r})} \\
&\leq M(\abs{\partial_{r_0}K(\widehat{r_0}, \widehat{r})}+\abs{\partial_r K(\widehat{r_0}, \widehat{r})})\abs{(r_0, r)-(\widetilde{r_0}, \widetilde{r})} \\
&\leq \widetilde{M}\abs{(r_0, r)-(\widetilde{r_0}, \widetilde{r})}.
\end{align*}
Here we used the fact that the point $(\widehat{r_0}, \widehat{r})$ belongs to the segment connecting $(r_0, r)$ and $(\widetilde{r_0}, \widetilde{r})$ so $\widehat{r}-\widehat{r_0}<\epsilon$ since $\Delta_R^\epsilon$ is a convex set. Therefore $K=K(r_0, r)$ is Lipschitz in a small neighborhood of the diagonal of~$\Delta_R$.
\end{proof}

Next we prove lemma~\ref{lemma:angularchange}. The proof is similar to the proof of lemma~\ref{lemma:abelkernel}.

\begin{proof}[Proof of lemma~\ref{lemma:angularchange}]
Using fundamental theorem of calculus we can write
$$
\omega(r_0, t)=\theta(r_0, t)-\theta_0=\int_0^t\dot{\theta}(r_0, s)\text{d}s=\int_{r_0}^r\dot{\theta}(r_0, u)\dot{r}^{-1}(r_0, u)\text{d}u.
$$
Since $\dot{r}^{-1}(r_0, t)=\bigoh(t^{-1})$ and $\dot{\theta}(r_0, t)=\bigoh(1)$ (by compactness and\NTR{Added compactness.} smooth dependence on initial conditions) we have by the chain rule
$$
\partial_r \omega(r_0, r)=\dot{\theta}(r_0, r)\dot{r}^{-1}(r_0, t)=\bigoh(1)\cdot\bigoh(t^{-1})=\bigoh(t^{-1}).
$$
Using $\dot{r}^{-1}=\bigoh(t^{-1})$ and $t^{-1}\approx (r-r_0)^{-1/2}$ for small~$t$ we obtain
$$
|\omega(r_0, r)|\leq M\int_{r_0}^r|\dot{r}^{-1}(r_0, u)|\text{d}u\leq \widehat{M}\int_{r_0}^r(u-r_0)^{-1/2}\text{d}u=\widetilde{M}(r-r_0)^{1/2}
$$
for small $r-r_0$ which implies $\omega(r_0, r)=\bigoh(t)$. Hence
$$
\partial_r\omega^2(r_0, r)=2\omega(r_0, r)\partial_r\omega(r_0, r)=\bigoh(t)\cdot\bigoh(t^{-1})=\bigoh(1).
$$
Thus the derivative $\partial_r\omega^2(r_0, r)$ is bounded for small~$t$ (or for small $r-r_0$).

For the derivative $\partial_{r_0}\omega(r_0, r)$ we write
\begin{align*}
\omega(r_0, r)=\int_{r_0}^r\dot{\theta}(r_0, u)(u-r_0)^{-1/2}K(r_0, u)\text{d}u=\int_{r_0}^r(u^2-r_0^2)^{-1/2}\varphi(r_0, u)\text{d}u
\end{align*}
where
$$
\varphi(r_0, r)=\dot{\theta}(r_0, r)K(r_0, r)(r+r_0)^{1/2}.
$$
The weak reversibility condition $\ddot{\theta}(r_0, 0)=0$ implies that
$$
\ddot{\theta}(r_0, t)=\int_0^t\dddot{\theta}(r_0, s)\text{d}s
$$
so $\ddot{\theta}(r_0, t)=\bigoh(t)$ since $\dddot{\theta}(r_0, t)=\bigoh(1)$ by compactness and\NTR{Added compactness.} smooth dependence on initial conditions. The transformation rules for the derivatives imply that
$$
\partial_r\dot{\theta}(y, r)=\dot{r}^{-1}(y, r)\ddot{\theta}(y, r)=\bigoh(t^{-1})\cdot\bigoh(t)=\bigoh(1)
$$
and
$$
\partial_y\dot{\theta}(y, r)=\partial_x\dot{\theta}(x, t)+\frac{\partial t}{\partial y}\ddot{\theta}(y, r)=\bigoh(1)+\bigoh(t^{-1})\cdot \bigoh(t)=\bigoh(1).
$$
Here we used the calculations from the proof of lemma~\ref{lemma:abelkernel} to deduce that $\frac{\partial t}{\partial y}=\bigoh(t^{-1})$ (see equation~\eqref{eq:timederivativey}) and $\partial_x\dot{\theta}(x, t)=\bigoh(1)$ by compactness and\NTR{Added compactness.} smooth dependence on initial conditions.
Thus $\dot{\theta}=\dot{\theta}(r_0, r)$ is Lipschitz for small $r-r_0$. Because $K=K(r_0, r)$ and $(r_0, r)\mapsto (r+r_0)^{1/2}$ are Lipschitz for small $r-r_0$ also the map $\varphi=\varphi(r_0, r)$ is Lipschitz for small $r-r_0$ (the terms in~$\varphi$ are bounded). Therefore the derivatives $\partial_{r_0}\varphi$ and $\partial_r\varphi$ are bounded for small $r-r_0$ and $\varphi=\bigoh(1)$.

The derivative $\partial_{r_0}\omega (r_0, r)$ becomes (see e.g.~\cite[Proposition 15]{deI:abel-transforms-x-ray-tomography}) 
\begin{align*}
\partial_{r_0}\omega (r_0, r)&=\int_{r_0}^r(u^2-r_0^2)^{-1/2}\bigg(\partial_{r_0}\varphi(r_0, u)-\frac{r_0}{u^2}\varphi(r_0, u)+\frac{r_0}{u}\partial_u\varphi(r_0, u)\bigg)\text{d}u \\
&-r_0(r^2-r_0^2)^{-1/2}\varphi(r_0, r).
\end{align*}
The term inside the big parenthesis is $\bigoh(1)$ because $\varphi$ and its derivatives are bounded and $u\geq r_0>R>0$. Hence the term coming from the integral is~$\bigoh(t)$. The latter term is $\bigoh(t^{-1})$ and thus $\partial_{r_0}\omega(r_0, r)=\bigoh(t^{-1})$. Finally we have
$$
\partial_{r_0}\omega^2(r_0, r)=2\omega(r_0, r)\partial_{r_0}\omega(r_0, r)=\bigoh(t)\cdot\bigoh(t^{-1})=\bigoh(1)
$$
which implies that the derivative $\partial_{r_0}\omega^2(r_0, r)$ is bounded for small~$t$, or equivalently for small $r-r_0$. The Lipschitz continuity of $(r_0, r)\mapsto \omega^2(r_0, r)$  in a small neighborhood of the diagonal of~$\Delta_R$ then follows from the mean value theorem as in the proof of lemma~\ref{lemma:abelkernel}.
\end{proof}

\section{Linearized travel time tomography on Finsler manifolds}
\label{sec:linearizedelastic}
In this section we consider the linearization of the boundary rigidity problem on Finsler manifolds: if two Finsler norms give the same distances between boundary points, are they equal up to a gauge? 
We show that the linearization of boundary distances for a general family of Finsler norms leads to the geodesic ray transform of a function on the sphere bundle~$SM$. We also show that if the family of Finsler norms arises from conformal variations, then linearization leads to the geodesic ray transform of scalar fields on~$M$. 
This implies that if the geodesic ray transform is injective on scalar fields, then we have boundary rigidity in the first order approximation.

The linearization of the boundary rigidity problem has been done earlier for Riemannian metrics for example in~\cite[Section 3.1]{SHA-ray-transform-riemannian-manifold}. We show that the Riemannian linearization result follows from the linearization of Finsler norms as a special case.

\subsection{Linearization for general Finsler norms}
Let~$M$ be a smooth compact manifold with boundary~$\partial M$.
Let $x, x'\in\partial M$, $\epsilon>0$ and $s\in (-\epsilon, \epsilon)$. Assume that we have a family of curves $\gamma_s\colon[0, T]\rightarrow M$ smoothly depending on~$s$ and connecting~$x$ to~$x'$ such that each~$\gamma_s$ is a unit speed geodesic of a Finsler norm~$F_s$. We denote by $\dot{\gamma}_s$ the derivative of~$\gamma_s=\gamma_s(t)$ with respect to~$t$. Let $d_{F_s}(x, x')$ be the length of the geodesic~$\gamma_s$ with respect to~$F_s$, i.e. we assume that~$\gamma_s$ minimizes the distance from~$x$ to~$x'$.

The derivative of $d_{F_s}(x, x')$ with respect to the parameter~$s$ at zero is
\begin{align}
\frac{\partial d_{F_s}(x, x')}{\partial s}\bigg|_{s=0}=\int_0^T\frac{\partial}{\partial s}(F_s(\gamma_s(t), \dot{\gamma}_s(t)))\bigg|_{s=0}\der t.
\end{align}
A calculation shows that
\begin{align*}
\frac{\partial}{\partial s}(F_s(\gamma_s(t), \dot{\gamma}_s(t)))\bigg|_{s=0}=\frac{\partial F_s(\gamma_0(t), \dot{\gamma}_0(t))}{\partial s}\bigg|_{s=0}+\frac{\partial}{\partial s}(F_0(\gamma_s(t), \dot{\gamma}_s(t)))\bigg|_{s=0}
\end{align*}
and we obtain
\begin{align*}
\frac{\partial d_{F_s}(x, x')}{\partial s}\bigg|_{s=0}&=\int_0^T\frac{\partial F_s(\gamma_0(t), \dot{\gamma}_0(t))}{\partial s}\bigg|_{s=0}\der t+\frac{\partial}{\partial s}\int_0^T F_0(\gamma_s(t), \dot{\gamma}_s(t))\der t\bigg|_{s=0}.
\end{align*}
The second term vanishes since~$\gamma_0$ is a geodesic of~$F_0$ and hence a critical point of the length functional. Thus we obtain
\begin{equation}
\frac{\partial d_{F_s}(x, x')}{\partial s}\bigg|_{s=0}=\geod_{SM} h(\gamma_0)
\end{equation}
where the function $h\colon SM\rightarrow \R$ is defined as
\begin{equation}
\label{eq:Fs-variation}
h(x, y)=\frac{\partial F_s(x, y)}{\partial s}\bigg|_{s=0}
\end{equation}
and~$\geod_{SM}$ is the geodesic ray transform on the sphere bundle~$SM$.

If $F_s(x, y)=\sqrt{g^s_{ij}(x)y^iy^j}$ where $g^s=g^s_{ij}(x)$ is a family of Riemannian metrics, then
\begin{align*}
\frac{\partial F_s(\gamma_0(t), \dot{\gamma}_0(t))}{\partial s}\bigg|_{s=0}&=\frac{1}{2\sqrt{g^0_{ij}(\gamma_0(t))\dot{\gamma}_0^i(t)\dot{\gamma}_0^j(t)}}\frac{\partial g^s_{ij}(\gamma_0(t))}{\partial s}\bigg|_{s=0}\dot{\gamma}_0^i(t)\dot{\gamma}_0^j(t) \\
&=\frac{1}{2}\frac{\partial g^s_{ij}(\gamma_0(t))}{\partial s}\bigg|_{s=0}\dot{\gamma}_0^i(t)\dot{\gamma}_0^j(t)
\end{align*}
where we used the fact that~$\gamma_0$ is a unit speed geodesic of~$g^0$. Hence 
\begin{align}
\frac{\partial d_{g_s}(x, x')}{\partial s}\bigg|_{s=0}=\geod_2 h(\gamma_0)
\end{align}
where the components of the 2-tensor field~$h=h_{ij}(x)$ are
\begin{equation}
\label{eq:h-Riemann}
h_{ij}(x)=\frac{1}{2}\frac{\partial g^s_{ij}(x)}{\partial s}\bigg|_{s=0}
\end{equation}
and~$\geod_2$ is the geodesic ray transform of 2-tensor fields.

If there are no constraints on the family~$F_s$ of Finsler geometries, then any smooth function $h\colon SM\to\R$ can be realized as a variation in the sense of equation~\eqref{eq:Fs-variation}.
Therefore the linearized problem in general Finsler geometry is that of finding the kernel of~$\geod_{SM}$.
This kernel characterization in the Finsler setting is surprisingly simple\NTR{Reworded to be more specific on what is simple. Cf. next paragraph.}: $\geod_{SM}h=0$ if and only if $h=Xu$ for a smooth function $u\colon SM\to\R$ with $u|_{\partial SM}=0$, where~$X$ is the geodesic vector field.
The claim can be proved by defining~$u$ to be the integral of~$h$ over forward geodesics.
Constraints on~$h$ induce constraints on the potential~$u$ as is the case in Riemannian linearizations and tomography of 2-tensors.

\NTR{Added this paragraph to compare the natures of various linearized travel time problems.}
If one studies the Riemannian version of linearized travel time tomography, then the deformation has the special form $h(x,y)=h_{ij}(x)y^iy^j$ with coefficients as in~\eqref{eq:h-Riemann}.
This structure of $h\colon SM\to\R$ (rank two tensor field) implies a special structure for $u\colon SM\to\R$ (rank one tensor field), and it is proving this structural implication that makes the Riemannian problem hard in comparison to the general Finslerian one.
See \cite{IM:integral-geometry-review,PSU-tensor-tomography-progress,SHA-integral-geometry-tensor-fields} for Riemannian results.
If one studies the linearized problem in the family of Finsler metrics arising from elasticity (see section~\ref{subsec:elastictraveltime}), the difficulty returns: one needs a structural characterization of possible variation fields~$h$ and the corresponding structure for the potential~$u$.
Neither of these structures is known in the general elastic setting.

\subsection{Linearization for conformal variations}
\label{sec:conformal}
Let us consider the case $F_s(x, y)=c_s(x) F_0(x, y)$ where $c_s=c_s(x)$ is a family of positive functions on~$M$ such that $c_0\equiv 1$ and~$F_0$ is some fixed Finsler norm. Now
\begin{align}
\frac{\partial F_s(\gamma_0(t), \dot{\gamma}_0(t))}{\partial s}\bigg|_{s=0}=\frac{\partial c_s(\gamma_0(t))}{\partial s}\bigg|_{s=0}
\end{align}
where we used the fact that $F_0(\gamma_0(t), \dot{\gamma}_0(t))=1$ since~$\gamma_0$ is a unit speed geodesic of~$F_0$. We obtain
\begin{align}
\frac{\partial d_{F_s}(x, x')}{\partial s}\bigg|_{s=0}=\geod f(\gamma_0)
\end{align}
where the function $f\colon M\rightarrow\R$ is defined as
\begin{equation}
f(x)=\frac{\partial c_s(x)}{\partial s}\bigg|_{s=0}.
\end{equation}
Hence the linearization of boundary distances of conformal family of Finsler norms leads to the geodesic ray transform of scalar fields on the Finsler manifold $(M, F_0)$. 

If all the geodesics~$\gamma_s$ give the same boundary distances $d_{F_s}(x, x')$, then the derivative with respect to~$s$ vanishes and
\begin{equation}
\geod f(\gamma_0)=0
\end{equation}
for all geodesics~$\gamma_0$ of~$F_0$ connecting two points on the boundary. If the geodesic ray transform is injective on $(M, F_0)$, then to first order in~$s$ we have
\begin{equation}
F_s\approx F_0+s\cdot\frac{\partial F_s}{\partial s}\bigg|_{s=0}=F_0
\end{equation}
since the derivative satisfies
\begin{equation}
\frac{\partial F_s(x, y)}{\partial s}\bigg|_{s=0}=\frac{\partial c_s(x)}{\partial s}\bigg|_{s=0}F_0(x, y)=f(x)F_0(x, y)=0
\end{equation}
where we used the fact that~$f=0$ whenever~$\geod$ is injective on $(M, F_0)$. In this case we have boundary rigidity up to first order in the parameter~$s$.

\subsection{Conformally linearized elastic travel time tomography}
\label{subsec:elastictraveltime}
Next we consider the travel time tomography problem in~$\R^3$ arising in elasticity. Basic theory of elasticity can be found for example in~\cite{CE-seismic-ray-theory, SHE-introduction-to-seismology}. The stiffness tensor $c_{ijkl}=c_{ijkl}(x)$ describes the elastic properties of a given material. The stiffness tensor has the symmetries
\begin{equation}
c_{ijkl}=c_{jikl}=c_{klij}.
\end{equation}
The density-normalized elastic modulus is 
\begin{equation}
a_{ijkl}(x)=\frac{c_{ijkl}(x)}{\rho(x)}
\end{equation}
where $\rho=\rho(x)$ is the density of the material. 

If~$p$ is the momentum covector, then the Christoffel matrix is $\Gamma_{il}(x, p)=\sum_{j, k}a_{ijkl}(x)p_jp_k$. The Christoffel matrix is symmetric and we also assume that it is positive definite so it has three positive eigenvalues $\lambda_i=\lambda_i(x, p)$ where $i\in\{1, 2, 3\}$. Let us assume that $\lambda_1>\lambda_i$ for $i\in\{2, 3\}$. It was shown in~\cite{deILS-finsler-boundary-distance-map} that $\sqrt{\lambda_1(x, p)}$ defines a co-Finsler norm in~$T^*\R^3$. Using the Legendre transformation we obtain a Finsler norm in~$T\R^3$.

Assume that the stiffness tensor $c_{ijkl}=c_{ijkl}(x)$ is fixed and consider the conformal variations $c^s_{ijkl}(x)=f_s(x)c_{ijkl}(x)$ where $f_s=f_s(x)$ is a smooth family of positive functions such that $f_0\equiv 1$ (i.e. we have a family of ``factorized anisotropic inhomogeneous media", see e.g.~\cite{CE-anisotropic-factorized, CE-anisotropic-factorized-2, CE-seismic-ray-theory, YN-finsler-seismic-ray-path}). The density-normalized elastic modulus becomes
\begin{equation}
a^s_{ijkl}(x)
=
\frac{c^s_{ijkl}(x)}{\rho(x)}
=
f_s(x)a_{ijkl}(x).
\end{equation}
Thus we have a family of Christoffel matrices $\Gamma^s_{il}(x, p)=f_s(x)\Gamma_{il}(x, p)$. Since the eigenvalues only get scaled by $f_s$, the largest eigenvalue corresponds to $\lambda^s_1(x, p)=f_s(x)\lambda_1(x, p)$. We obtain a family of co-Finsler norms $F_s^*(x, p)=\sqrt{f_s(x)}F^*(x, p)$ where~$F^*$ is the co-Finsler norm corresponding to the stiffness tensor~$c_{ijkl}$. Now as the Legendre transformation acts fiberwise, we obtain a family of Finsler norms $F_s(x, y)=\sqrt{f_s(x)}F(x, y)$ where~$F$ is the Legendre transformation of~$F^*$. 

We have shown that conformal variations of the stiffness tensor leads to conformal variations of the Finsler norm induced by the background stiffness tensor. If we consider the travel time tomography or the boundary rigidity problem for the family of induced Finsler norms~$F_s$, then using the observations done in section~\ref{sec:conformal} we obtain that 
\begin{equation}
\frac{\partial f_s(x)}{\partial s}\bigg|_{s=0}=0
\end{equation}
whenever the geodesic ray transform is injective on scalar fields on the base manifold $(M, F)$. This in turn implies that the stiffness tensors~$c^s_{ijkl}=f_sc_{ijkl}$ all agree to first order in~$s$, i.e.
\begin{equation}
c^s_{ijkl}
\approx
c_{ijkl}+s\cdot\frac{\partial c^s_{ijkl}}{\partial s}\bigg|_{s=0}
=
c_{ijkl}.
\end{equation}

We note that it was shown in~\cite{dESUZ-generic-uniqueness-mixed-ray} that the linearization of the elastic travel time tomography problem for a family of isotropic stiffness tensors leads to the geodesic ray transform of scalar fields on Riemannian manifolds (and more generally to an integral geometry problem of 4-tensor fields). Our conformal linearization allows general anisotropies for~$c_{ijkl}^s$ (the background stiffness tensor~$c_{ijkl}$ can be anisotropic) and therefore the geometry is Finslerian; in ~\cite{dESUZ-generic-uniqueness-mixed-ray} the authors mainly study perturbations around isotropic elasticity (weakly anisotropic medium). Other difference is that our linearization applies to qP-waves and the linearization in~\cite{dESUZ-generic-uniqueness-mixed-ray} to S-waves and qS-waves. For more linearization results in elastic travel time tomography see~\cite{dESUZ-generic-uniqueness-mixed-ray, deSZ-mixed-ray, SHA-integral-geometry-tensor-fields}.\NTR{Updated references (book by Paternain, Salo and Uhlmann and article by Kurusa and Ódor). Also removed an unpublished article by Lehtonen and added references~\cite{HMS:attenuated, PSU:attenuated} on the attenuated ray transform.}

\bibliography{refs} 
\bibliographystyle{abbrv}
\end{document}